\documentclass[11pt,reqno]{amsart}
\numberwithin{equation}{section}
\usepackage[dvipsnames]{xcolor}
\usepackage{tikz}
\usetikzlibrary{snakes}
\usetikzlibrary{patterns}
\usepackage{amsmath,amsfonts,amssymb,amsbsy,amsopn,amsthm,eucal, mathrsfs}
\usepackage{enumerate, enumitem, esint}
\usepackage{dsfont}
\usepackage{graphicx}   
\usepackage[bookmarks=true, pdfstartview=FitH, colorlinks=true,citecolor=blue, linkcolor=blue]{hyperref}
\usepackage{accents}
\usepackage{color}

\usepackage{cite}

\usepackage{etoolbox}

\makeatletter
\patchcmd{\@tocline}
{\hfil}
{\dotfill}
{}{}

\renewcommand{\l@subsection}{\@tocline{2}{0pt}{2.5pc}{5pc}{}}
\makeatother

\newcommand{\Ric}{{\rm Ric}}

\newcommand{\diam}{{\rm diam}}

\newcommand{\Alex}{\text{Alex\,}}

\newcommand{\Alexnk}{\text{Alex}^n(\kappa)}

\newcommand{\dN}{\mathds{N}}

\newcommand{\dR}{\mathds{R}}

\newcommand{\cH}{\mathcal{H}}

\newcommand{\cK}{\mathcal{K}}

\newcommand{\cN}{\mathcal{N}}

\newcommand{\cS}{\mathcal{S}}

\newcommand{\Vol}[1]{\text{Vol}\left(#1\right)}

\newcommand{\dsp}{\displaystyle}

\newcommand{\RNum}[1]{\uppercase\expandafter{\romannumeral #1\relax}}

\newtheorem*{Theorem A}{Theorem A}
\newtheorem*{Theorem B}{Theorem B}
\newtheorem*{Theorem C}{Theorem C}
\newtheorem*{Theorem D}{Theorem D}
\newtheorem*{Theorem E}{Theorem E}

\newtheorem{theorem}{Theorem}[section]

\newtheorem{lemma}[theorem]{Lemma}
\newtheorem{corollary}[theorem]{Corollary}
\newtheorem{conjecture}[theorem]{Conjecture}
\newtheorem{question}[theorem]{Question}
\newtheorem{definition}[theorem]{Definition}
\newtheorem{example}[theorem]{Example}
\newtheorem{remark}[theorem]{Remark}

\makeatletter
\def\smalloverbrace#1{\mathop{\vbox{\m@th\ialign{##\crcr\noalign{\kern3\p@}%
  \tiny\downbracefill\crcr\noalign{\kern3\p@\nointerlineskip}%
  $\hfil\displaystyle{#1}\hfil$\crcr}}}\limits}
\makeatother

\begin{document}

\title{Bounding Curvature Measure on Manifolds with Singularities}

\author[N.~Li]{Nan Li}

\address{Department of Mathematics, Capital Normal University, Beijing, P.R.C.}
\address{Department of Mathematics, The City University of New York - NYC College of
	Technology, 300 Jay St., Brooklyn, NY 11201}
\email{NLi@citytech.cuny.edu}
\date{\today}

\thanks{Nan Li was partially supported by PSC-CUNY \#66371‐00 54.}

\maketitle

\begin{center}
\end{center}

\begin{abstract}
  Let $X$ be an $n$-dimensional Alexandrov space with curvature $\ge -1$, and let $\eta > 0$. Define $\mathcal{S}^{k}_\eta(X)$ as the set of $(k,\eta)$-singular points in $X$ whose tangent cones are $\eta$-away from  splitting off $\mathbb{R}^{k+1}$ isometrically. For a point $p \in X$, assume that $M = B_2(p) \setminus (\mathcal{S}^{n-2}_\eta(X) \cup \partial X)$ is a smooth manifold equipped with the Riemannian metric induced by $X$. We prove that the integral of the scalar curvature of $M$ over $B_1(p)$ is bounded from above by a constant depending only on $n$ and $\eta$. As a special case, this extends Petrunin's bounded curvature integral result for complete manifolds with lower sectional curvature bound to the setting of open manifolds and smooth manifolds with boundary, provided that these manifolds are Alexandrov spaces.
\end{abstract}

\tableofcontents

\section{Introduction}\label{s:intro}

There are three motivations for this work. The first is the relation to the following well-known problem proposed by Yau, which asks whether a lower bound on the Ricci curvature is sufficient to control the integral of scalar curvature.

\begin{question}[Problem 9, \cite{Yau92}]\label{c:Yau}
	Let $(M, g)$ be an $n$-dimensional complete Riemannian manifold with Ricci curvature $\Ric_M\ge 0$ and $p\in M$. Does it hold that 
	\begin{align}
		\lim_{R\to\infty}R^{-(n-2)}\int_{B_R(p)}scal \,\operatorname d {vol}_{g}<\infty?
	\end{align}
\end{question}

In general, it was conjectured that if $\Ric_M\ge -1$, then 
\begin{align}
	\int_{B_1(p)}scal \,\operatorname d {vol}_{g}<C(n).
\end{align} 

This conjecture is confirmed if the manifold has a lower sectional curvature bound, while the Ricci case remains open. 

\begin{theorem}[Petrunin \cite{Pet09}]\label{t:Pet.X}
	For any $n$-dimensional complete Riemannian manifolds $(M, g)$ with sectional curvature $\sec_M\ge -1$, we have an a prior $L^1$-bound for the scalar curvature
	$\dsp\int_{B_1}scal \,\operatorname d {vol}_{g}\le c(n)$.
\end{theorem}

A detailed direct proof of this result, avoiding an argument by contradiction, was provided in \cite{Li25}. As the main result, we generalize Theorem \ref{t:Pet.X} to open manifolds and Alexandrov spaces, where Toponogov's comparison theorem holds for any geodesic triangle \cite{BGP}. Let $X \in \Alexnk$ be an $n$-dimensional Alexandrov space with curvature $\ge\kappa$. We let $\cS^{k}(X)$ denote the set of $k$-singular points of $X$, whose tangent cones do not split off $\dR^{k+1}$ isometrically. Furthermore, let $\cS^{k}_\eta(X)$ denote the set of $(k,\eta)$-singular points, whose tangent cones are $\eta$-away from splitting off $\dR^{k+1} $isometrically. It is well known that $\cS^{n-1}(X)\setminus \cS^{n-2}(X)\subseteq\partial X$, and the Hausdorff dimension satisfies $\dim_{\cH}(\cS(X))\le n-1$. For more detailed definitions and Hausdorff measure estimates regarding these quantitative singular sets, we refer the reader to \cite{LiNab20}. The following is our main result. 

\begin{Theorem A}\hypertarget{thm:A}{}
	Let $(X, p, d_X)\in\Alex^n(-1)$ and $\eta>0$. Suppose that $B_2(p)\setminus(\cS^{n-2}_\eta(X)\cup\partial X)$ is a smooth Riemannian manifold, equipped with Riemannian metric $g$ induced by $d_X$. Then 
	\begin{align}
		\int_{B_1(p)}scal \,\operatorname d {vol}_{g}\le C(n, \eta).
		\label{e:thmA-1}
	\end{align}
	In particular, the above estimate holds on smooth manifold with boundary and smooth open manifold if they (or their metric completions) are Alexandrov spaces. 
\end{Theorem A}

\begin{remark}
	\hyperlink{thm:A}{Theorem A} is new for open manifold or manifold with boundary. This statement does not follow directly from Theorem \ref{t:Pet.X} or its proof. Technically, this is because open manifolds lack the controlled covering structure used in \cite{Pet09} or \cite{Li25}. In fact, if an open manifold $M$ is Alexandrov, then one can apply \hyperlink{thm:A}{Theorem A} on its metric completion and obtain the desired result. 
\end{remark}

The following example shows that \hyperlink{thm:A}{Theorem A} doesn't hold without the (global) Alexandrov condition.

\begin{example}\label{r:eg.bdy}
    Let $H_r$ denote a $2$-dimensional sphere of radius $r \ge 1$, from which the closed $(r/10)$-balls centered at the north and south poles have been removed. Let the open manifold $M$ be the universal cover of $H_r$, and let $X$ be the metric completion of $M$. Note that though $M$ is a smooth open manifold satisfying $\sec_M \ge \frac{1}{r^2} \ge 1$, neither $M$ nor $X$ is an Alexandrov space. In this case, for any point $p \in M$ or $X$, the following holds:
    $$ \displaystyle\int_{B_1(p)} \operatorname{scal} \, \operatorname d {vol}_{g} \ge r^{-1} \to \infty \quad \text{as } r \to 0. $$
    
    As demonstrated in our proof, the boundary of an Alexandrov space has controlled bending, a property that is essential for bounding the curvature integral.
\end{example}

\begin{remark} 
	The dependence on $\eta$ cannot be removed in our approach. This is partly because our proof relies on a careful analysis of the distribution of the singular sets, as they contribute to the boundary terms of the integrals. This analysis can be carried out for the singular set $\cS_\eta(X)$ if the fineness parameter $\epsilon$ of the covering is chosen to be much smaller than $\eta$. However, the smaller $\epsilon$ is chosen, the greater the number of elements needed in the covering. We believe this dependence can be removed, but it would rely on a finer estimate of the total singularity. In particular, one must carefully address the case in which the singular set $\cS(X)$ is dense (see Example (2) in \cite{OtsShi94}).
\end{remark}

During the course of proving \hyperlink{thm:A}{Theorem A}, the author also tried to show that the integral of the intrinsic scalar curvature of $\partial X$ is bounded. It seems that the same type of argument should work, but unfortunately, it fails due to some technique reasons.

\begin{conjecture}\label{c:bdy-curv}
	Let $(X, d_X)\in\Alex^n(-1)$ and $p\in\partial X$. Suppose that $B_2(p)\cap\partial X\setminus\cS^{n-3}(X)$ is a smooth Riemannian manifold, equipped with Riemannian metric $g_{\partial X}$ induced by $d_X$. Then 
	\begin{align}
		\int_{B_1(p)\cap\,\partial X}scal_{\partial X} \,\operatorname d {vol}_{g_{\partial X} }\le C(n).
	\end{align}
\end{conjecture}

It is conjectured that the boundary of an Alexandrov space $X$ is also an Alexandrov space. Currently, the only known case where this holds is when $X$ is smooth \cite{AKP08}. Conjecture \ref{c:bdy-curv} can be viewed as support for this boundary conjecture. Furthermore, Example \ref{r:eg.bdy} provides evidence for Conjecture \ref{c:bdy-curv} by demonstrating that a (global) Alexandrov condition is necessary to bound the curvature integral.

The second motivation is the following conjecture proposed by Naber. Suppose that $(M_i,g_i)$ Gromov-Hausdorff converges to a limit space $(X, d)$, and $scal_i$ has a uniform integral bound as stated in Theorem \ref{t:Pet.X}. Passing to a sub-sequence, as a signed measure, $scal_i \,\operatorname d {vol}_{g_i}$ converges to a measure on $X$. 

\begin{conjecture}[Conjecture 2.20 in \cite{Nab20}]\label{c:measure conv}
	Let $(M_i, g_i)$ be a sequence of $n$-dimensional manifolds with lower Ricci or sectional curvature bounds. If $M_i$ Gromov-Hausdorff converge to $X$ without collapsing, then as measure, 
	$$ scal_i \,\operatorname d {vol}_{g_i} \to \mu = {\operatorname R}\,\operatorname d \cH^n + \Phi\,\operatorname d \cH^{n-1} +\theta \,\operatorname d \cH^{n-2},
	$$
	where ${\operatorname R}$, $\Phi$ and $\theta$ are locally $L^1$-functions with respect to $\cH^n$, $\cH^{n-1}$ and $\cH^{n-2}$ respectively. The function $\Phi(x)$ is supported on an $(n-1)$-rectifiable subset. The function $\theta(x)$ is supported on the top stratum of the $(n-2)$-rectifiable singular set $\cS(X)=\cS^{n-2}(X)$.
\end{conjecture}

For the sectional curvature case, some progress has been made on this conjecture, particularly concerning the singular sets of co-dimension one and two \cite{LiNab20}, \cite{Li23} and \cite{LPet22}. 

\begin{theorem}[\cite{LPet22}]\label{t:cont}
	Let $M_i$ be a sequence of $n$-dimensional Riemannian manifold with $\sec_{M_i}\ge\kappa$. Suppose that $M_i$ Gromov-Hausdorff converge to $X$ without collapsing. Then the curvature
	tensors of $M_i$ weakly converge to a measure-valued tensor $\mu$ on $X$ which satisfies the following properties.
	\begin{enumerate}
		\renewcommand{\labelenumi}{(\roman{enumi})}
		\item If $p\in X$ is smooth, then $\mu(p)$ is equivalent to ${\operatorname R}(p)\,\operatorname d \cH^n$, where ${\operatorname R(p)}$ is the curvature tensor induced by the metric of $X$. 
		\item If the tangent cone $T_p(X)=C(\Sigma_p)\times \mathds R^{n-2}$, then $\mu(p)$ is equivalent to $(2\pi-\beta(p))\,\operatorname d \cH^{n-2}$, where $\beta(p)$ is the cone angle of $C(\Sigma_p)$. 
		\item The measure $\mu$ vanishes on any subset of $X$ with a vanishing $(n-2)$-dimensional Hausdorff measure. In particular, If $T_p(X)$ doesn't splits off $\mathds R^{n-2}$ isometrically, then $\mu(p)=0$. 
	\end{enumerate}
\end{theorem}

The above result does not provide an intrinsic description for the term $\Phi\,\operatorname{d}\cH^{n-1}$, although examples (such as the doubled disk) show that it exists and may behave like the mean curvature of hypersurfaces. By definition, the support of $\Phi$ should include non-$C^2$ points, which may have codimension exactly $1$  (see, e.g., doubled disk). However, it was questionable whether the set of non-smooth points in an Alexandrov space could have codimension less than $1$. 

Let $\mathcal{N}(X)$ denote the collection of non-$C^2$ points in $X \in \Alexnk$. Recall that a perfect set is a closed set containing no isolated points. For example, closed intervals and Cantor sets are perfect sets. The following result is well known to experts based on the famous Alexandrov Existence Theorem for the $2$-dimensional Monge-Amp\`ere equation. For reader's convenience, we will give an explicit constructive proof in this paper.

\begin{theorem}\label{thm:B}
	For any perfect set $\mathcal{K}\subset (0,1)$ with Hausdorff dimension $\alpha\in[0,1]$, there exists an Alexandrov space $X\in\Alex^2(0)$ such that the following hold.
	\begin{enumerate}
		\renewcommand{\labelenumi}{(\arabic{enumi})}
		\setlength{\itemsep}{1pt}
		\item $X$ is homeomorphic to a disk in $\mathbb R^2$. 
		\item $X$ is flat outside $\mathcal N(X)$. 
		\item $\mathcal N(X)\setminus\partial X$ is bi-Lipschitz equivalent to $\mathcal{K}\times \mathcal{K}$.
		\item The curvature measure $\mu$ is supported on $\mathcal N(X)$. 
	\end{enumerate}
	Moreover, if $\mathcal{K}$ is Ahlfors regular, then the Hausdorff dimension $\dim_{\cH}(\mathcal N(X)\setminus\partial X)=2\alpha$ and $\mu|_{\mathcal N(X)\setminus\partial X}$
	is equivalent to the Hausdorff measure $\mathcal H^{2\alpha}$.
\end{theorem}

In particular, $\mathcal{N}(X)\setminus\partial X$ can be a Cantor set with arbitrarily small fractal codimension, and make a positive contribution to the curvature measure.
These examples indicate that, in Naber's Conjecture \ref{c:measure conv}, extra terms are needed to handle this fractal type of curvature distribution.

In the context of Alexandrov spaces, classical curvature values are not well-defined at non-smooth points. Although we have identified examples not covered by Conjecture \ref{c:measure conv}, \hyperlink{thm:A}{Theorem A} and Theorem \ref{t:cont} strongly suggest the existence of a (signed) locally finite synthetic curvature measure on general Alexandrov spaces. It is a well-established fact that if a sequence of Riemannian manifolds $(M_i, g_i)$ satisfying $\sec_{M_i}\ge \kappa$ converges to a space $X$ in the Gromov-Hausdorff sense, then $X\in\Alexnk$. Provided this convergence is non-collapsing (in which case $X$ is termed smoothable), a curvature measure on $X$ can be explicitly defined via Theorem \ref{t:cont}. However, this definition remains extrinsic for non-smooth regular points on $X$, as Theorem \ref{t:cont} lacks an intrinsic description for this particular stratum. Furthermore, as a consequence of Perelman's Stability Theorem, not all Alexandrov spaces can be realized as non-collapsed limits of Riemannian manifolds with uniform lower sectional curvature bounds. To address these most general cases, we propose the following conjecture.

\begin{conjecture}\label{c:curv-measure}
	Let $X\in \Alexnk$. Then there exists a locally finite curvature measure 
	$$ \mu = {\operatorname R}\,\operatorname{d}\!\cH^n + \theta\,\operatorname{d}\!\cH^{n-2} + \omega, $$
	where ${\operatorname R}$ and $\theta$ are locally $L^1$-functions with respect to $\cH^n$ and $\cH^{n-2}$ respectively. The function ${\operatorname R}(x)$ is the scalar curvature, supported on the $C^2$ part. The function $\theta(x)$ is supported on the top stratum of the $(n-2)$-rectifiable singular set $\cS(X) = \cS^{n-2}(X)$. The term $\omega$ is supported on the union of non-$C^2$ regular points and the boundary. 
		
	Moreover, if $X$ is a non-collapsed limit of a sequence of Riemannian manifolds with a uniform lower sectional curvature bound, then such a measure is equivalent to the limit scalar curvature measure.
\end{conjecture}

Note that \hyperlink{thm:A}{Theorem A} partially verifies the this conjecture for Alexandrov spaces under specific additional assumptions. Furthermore, Example \ref{r:eg.bdy} and Theorem \ref{thm:B} demonstrate why general Alexandrov spaces present a significant challenge. Even to establish the local finiteness of the absolutely continuous part ${\operatorname R}\,\operatorname{d}\!\cH^n$, one need to precisely understand the distribution of non-smooth regular points. Indeed, these points contribute directly to the boundary terms in our integral estimates (see, e.g., \eqref{l:int.corner.mfd.e9.6.5}). Unfortunately, the nature of this distribution remains poorly understood at the present stage.

As our third motivation and a further application, \hyperlink{thm:A}{Theorem A} yields a singular version of the following conjecture proposed by Gromov.

\begin{conjecture}[\cite{Gro86}, Open Question 2.A.(b)]
	There exists a dimensional constant $C(n)$ such that if $(M, g)$ is an $n$-dimensional Riemannian manifold with $Ric_M\ge 0$ and $scal_M\ge 1$, then for any $p\in M$, we have
	\begin{align}
		\sup_{R>0}\frac{vol(B_R(p))}{R^{n-2}}\le C(n).
	\end{align}
\end{conjecture}
Gromov outlined a proof of this conjecture for the case of nonnegative sectional curvature (see \cite{Gro86}) and raised the question of whether it could be extended to the nonnegative Ricci curvature setting. Observe that when $\sec_M\ge 0$, a straightforward rescaling argument applied to Theorem \ref{t:Pet.X} yields $\dsp\int_{B_R(p)}scal \,\operatorname d {vol}_{g}\le c(n)R^{n-2}$. This resolves the conjecture in the nonnegative sectional curvature case. Notably, this method differs from Gromov's original sketch. Utilizing a similar rescaling technique, \hyperlink{thm:A}{Theorem A} implies the following result regarding the collapsing of non-negatively curved Riemannian manifolds that possess boundaries and singularities.

\begin{corollary}\label{c:thmA-large}
	Let $M\in\Alex^n(0)$ be an open manifold or smooth manifold with boundary.  Then for any $p\in M$ and $R>0$, we have
	\begin{align}
		\int_{B_R(p)}scal \,\operatorname d {vol}_{g}\le C(n)R^{n-2}.
	\end{align}
	If in addition, $scal_x\ge 1$ for any interior point $x\in M\setminus\partial M$, then 
	\begin{align}
		\sup_{R>0}\frac{vol(B_R(p))}{R^{n-2}}\le C(n).
		\label{c:thmA-large.e1}
	\end{align}
\end{corollary}

Example \ref{r:eg.bdy} also gives a counterexample to (\ref{c:thmA-large.e1}) when $M$ is a non-negatively curved open manifold that does not belong to $\Alex^n(0)$.

Our approach is inspired by that of Petrunin \cite{Pet09}. The primary difference is that we utilize a finite stratified covering directly on the given space (or its metric completion), rather than arguing by contradiction on a sequence of manifolds and intersections of level sets that satisfy certain technical lifting properties. Additionally, we need to modify the covering annuli to handle singularities and boundary terms.

The outline of the proof is as follows. Using the covering technique developed in \cite{LiNab20}, we first cover the given unit ball $B_1$ by a controlled number of compact subsets that behave like annular regions in metric cones. This is called a good scale annulus covering. Using the almost metric cone structure and the Gauss--Codazzi formula, we reduce the estimate of the scalar curvature integral on $B_1$ to the estimates of the intrinsic scalar curvature integrals on the level sets of the covering annuli. Then, we cover the level set of each annular region by a controlled number of good scale annuli again. By adjusting the intersection angles between the level sets from different strata and the boundaries, we can repeat this process and apply mathematical induction on the dimension of the intersection of these level sets (called corner spaces). We also need to carefully analyze the distribution of singular points on each of the corner spaces in order to control the boundary terms. The $2$-dimensional case follows from the finiteness of the number of components of the corner spaces.

To conclude the introduction, we present some open questions in this direction. By a straightforward rescaling argument, it follows from Theorem \ref{t:Pet.X} that if $\sec_M\ge\kappa$, then $\dsp\int_{B_r}scal \,\operatorname d {vol}_{g}\le c(n,\kappa)r^{n-2}$, provided $r\le 1$ when $\kappa <0$. It is easy to see that as $r\to 0$, the sharp constant $c(n,\kappa)$ approaches the same constant as in the Euclidean case. However, for a fixed $r>0$, we cannot determine the sharp constant $c(n,\kappa)$ merely by tracing our covering-type proof.

\begin{question}
	Let $(M, g)$ be an $n$-dimensional manifold with $\sec_M\ge \kappa$. What is the sharp upper bound of $\dsp\int_{B_r}scal \,\operatorname d {vol}_{g}$? 
\end{question}

\begin{conjecture}
	Let $(X, d_X)$ be an $n$-dimensional Alexandrov space with curvature $\ge 1$. Let $g$ be the Riemannian metric induced by $d_X$ on the smooth part $X_1$.  
	\begin{enumerate}
		\renewcommand{\labelenumi}{(\roman{enumi})}
		\item If $\partial X=\varnothing$, then $\dsp\int_{X_1}scal(x) \,\operatorname d {vol}_{g}\le \int_{\mathds S^n_1} scal_{_{\mathds S^n_1}}\operatorname d {vol}_{S^n_1}$, where $\mathds S^n_1$ is the $n$-dimensional standard sphere. The equality holds if and only if $X$ is isometric to $\mathds S^n_1$, provided $n\ge 3$.
		\item If $\partial X\neq \varnothing$, then $\dsp\int_{X_1}scal(x) \,\operatorname d {vol}_{g}\le \frac12\int_{\mathds S^n_1} scal_{_{\mathds S^n_1}}\operatorname d {vol}_{S^n_1}$. The equality holds if and only if $X$ is isometric to the $n$-dimensional standard semi-sphere, provided $n\ge 3$.
	\end{enumerate}
\end{conjecture}

The boundary case can be proved by an easy doubling argument if the case without boundary is true. A similar conjecture was proposed for the smooth setting in \cite{Li25}.

The author would like to thank Capital Normal University for the host when completing this work. He also thanks Anton Petrunin, Tadashi Fujioka and Ge Xiong for helpful discussions.

\section{$\epsilon$-Frame Covering}\label{sec:frame-covering}

We develop a stratified covering technique in this section. Each element in our covering will be an annulus-like region. In the next section, we will calculate the integral of scalar curvature by induction along this stratified structure. 

\subsection{Good scale annulus covering}

Let $X\in\Alex^n(-1)$, $p\in X$ and $A_a^b(p)=\{x\in X\colon a\le d(p,x)\le b\}$ denote the closed annulus centered at $p$. In particular, if $a=0$, then $A_0^r(p)=\bar B_r(p)$ is the closed geodesic ball. 

\begin{definition}[Good scale annulus]
	The annulus region $A_a^b(p)$ is said to be an $(\epsilon, \sigma)$-good scale annulus if the following hold. 
	\begin{enumerate}
		\renewcommand{\labelenumi}{(\roman{enumi})}
		\item $a\le\sigma^6 b$.
		\item There exists a metric cone $C_{p^*}(\Sigma)$ such that
		\begin{align}
			d_{GH}\Big(B_s(p), \, B_s(p^*)\Big)\le \epsilon^2 s
			\label{defn:bad scale.e1}
		\end{align}
		for every $s\in[\sigma^6 a,\sigma^{-6}b]$.
	\end{enumerate}
\end{definition}

Any compact set in $X$ can be covered by a controlled number of good scale annuli. This is known as the Annulus Covering Theorem.

\begin{theorem}[Covering by good scale annuli, Theorem 15.3.14 \cite{Li20}, Li-Naber \cite{LiNab20}]\label{t:ann_cover}
	Let $\epsilon, \sigma >0$ and $n\in\dN$. 
	For any $X\in\Alex^n(-1)$ and any closed subset $\Omega\subseteq\bar B_1$, there exists a collection of $(\epsilon, \sigma)$-good scale annuli  $\mathcal C=\Big\{A_{a_i}^{b_i}(p_i)\Big\}$ such that  
	\begin{enumerate}
		\renewcommand{\labelenumi}{(\roman{enumi})}
		\item $\Omega\subseteq \cup_iA_{a_i}^{b_i}(p_i)$,
		\item $|\mathcal C|<C(n,\epsilon,\sigma)$,
		\item $p_i\in\Omega$ for every $i$.
	\end{enumerate}
\end{theorem}

Note that the lower bound for the size of the covering annuli depends on $\Omega$ and $X$ and it can be arbitrarily small. This is different from the classical covering by finitely many balls whose radii usually have a  uniform positive lower bound.  

\subsection{$\epsilon$-frames}\label{subsec:eps-frame}

Let $\Big\{A_{a_i}^{b_i}(p_i)\Big\}$ be an $(\epsilon, \sigma)$-good scale annulus covering of $B_1 \subseteq X$. Note that an annulus region $A_a^b(p)$ can be exhausted by the level sets of its distance function: 
$$A_a^b(p) = \bigcup_{a \le t \le b} d_p^{-1}(t).$$
Fix $i$ and $t_i \in [a_i, b_i]$. By the Annulus Covering Theorem (Theorem \ref{t:ann_cover}), there is an $(\epsilon, \sigma)$-good scale annulus covering $\Big\{A_{a_j'}^{b_j'}(q_j)\Big\}$ for $\Omega = d_{p_i}^{-1}(t_i)$. By inductively applying this covering theorem to the intersections of level sets from the descending stratum, we obtain a coordinate-like frame for each point in the intersections. In our application, we also need to perturb the level sets such that they carry more geometric properties.

\begin{definition}[$f$-adjusted good scale annulus]\label{d:cvx.ann.mfd} 
	Let $A_{a}^{b}(p)$ be an $(\epsilon,\sigma)$-good scale annulus. Let $f\colon X\to \dR$ be a semiconcave function. The annulus region $W_a^b(p)=\cup_{a\le t\le b} \, f^{-1}(t)$ is said to be an  $(\epsilon,\sigma, \delta, f)$-good scale annulus, or simply $f$-adjusted or $(\epsilon, \sigma, \delta)$-adjusted annulus, denoted by $(W_a^b(p), f)$, if the following hold for every $x\in W_{\sigma^4a}^{\sigma^{-4}b}(p)$.
	\begin{enumerate}
		\renewcommand{\labelenumi}{(\roman{enumi})}
		\item $\left|f(x)-d_p(x)\right|\le \delta\cdot d_p(x)$.
		\item $\left||\nabla_x f|-1\right|<\delta$.
		\item $|\left\langle \nabla_x f, \nabla_xd_{p}\right\rangle-1| <\delta$.
	\end{enumerate}
\end{definition}

By a slight abuse of notation, we may denote $W_a^b(p)$ by $W$ and denote $f^{-1}(t)$ by $\partial W_t(p)$ or simply $\partial W_t$, where $\sigma^4a\le t\le \sigma^{-4}b$. Note that every $(\epsilon,\sigma)$-good scale annulus is a $(\epsilon,\sigma, 10\epsilon, d_p)$-good scale annulus. Next, we give the notions of frames.

\begin{definition}[$\epsilon$-frame]\quad
	\begin{enumerate}
		\renewcommand{\labelenumi}{(\roman{enumi})}
		\item A sequence of level sets $\{\partial W^1_{t_1}(p_1), \partial W^2_{t_2}(p_2),\dots, \partial W^m_{t_m}(p_m)\}$ is called an $(\epsilon, \sigma, \delta, f_1, \dots, f_m)$-frame if the following conditions are satisfied: 
		\begin{enumerate}
			\renewcommand{\labelenumii}{(\alph{enumii})}
			\item For every $i = 1, 2, \dots, m$, $W_{a_i}^{b_i}(p_i)$ is an $(\epsilon, \sigma, \delta, f_i)$-good scale annulus, and $\partial W^i_{t_i}=f_i^{-1}(t_i)$ with $t_i\in[a_i, b_i]$;
			\item For every $i = 1, 2, \dots, m-1$, $p_{i+1}\in \bigcap_{j=1}^i\partial W^j_{t_j}$.
		\end{enumerate}
		The intersection $Y=\bigcap_{i=1}^m \partial W^i_{t_i}$ is called an $(m,\epsilon, \sigma, \delta)$-corner space with respect to the defining functions $(f_1, f_2, \dots, f_m)$. We let $Y=X$ if $m=0$. 
		\item Such a frame (or corner space) is said to be $(\delta_1, \delta_2)$-orthogonal if the following hold: 
		\begin{enumerate}
			\renewcommand{\labelenumii}{(\alph{enumii})}
			\item $\delta_1 \le \langle \nabla_x f_j, \nabla_xf_k \rangle \le \delta_2$ for every $j\neq k$ and $x\in Y$;
			\item $-\delta_2 \le \langle \nabla_x f_j, \nabla_xd_{\partial X} \rangle \le -\delta_1$ for every $x\in Y\cap \partial X$. 
		\end{enumerate}
		\item A $(\delta_1, \delta_2)$-orthogonal frame $Y$ (or corner space) is said to be acute if $\delta_1 <\delta_2 < 0$. That is, the intersection angles between the level sets are acute. $Y$ is called a.e.-acute if condition (ii)-(b) holds for $\mathcal{H}^{n-m-1}$-almost every $x\in Y\cap\partial X$. 
		\item In our application, parameters $\sigma, \delta, \delta_1$, and $\delta_2$ are all of the form $c(n)\epsilon$. We may combine the dependent parameters and simply call the above frames an $(m,\epsilon)$-frame, an $(m,\epsilon)$-orthogonal frame, and an $(m,\epsilon)$-acute frame, respectively. Corner spaces follow the same renaming convention. 
	\end{enumerate}
\end{definition}

\subsection{$\epsilon$-frame covering}

The frames must first be $\epsilon$-orthogonal to serve as a replacement for local coordinates. We then take such an $\epsilon$-frame and perturb it into an a.e.-acute $\epsilon$-frame. This acuteness is necessary for controlling the second fundamental form. See \eqref{l:int.corner.mfd.e22}.

Indeed, the inductive covering procedure described in Section \ref{subsec:eps-frame} naturally yields an $\epsilon$-orthogonal frame under a mild additional condition.

\begin{lemma}\label{lem:ortho-scale}
	Let $\mathcal F = \{\partial W^1_{t_1}(p_1), \partial W^2_{t_2}(p_2),\dots, \partial W^m_{t_m}(p_m)\}$ be an $(\epsilon, \delta, f_1, \dots, f_m)$-frame. If $0<t_{i+1}\le\epsilon t_i$, then $\mathcal F$ is $\epsilon$-orthogonal. 
\end{lemma}
\begin{proof}
	By Definition \ref{d:cvx.ann.mfd}, it suffices to prove this for $f_i=d_{p_i}$. Let $A_{a_i}^{b_i}(p_i)$ and $A_{a_k}^{b_k}(p_k)$ be the good scale annuli defined by $f_i=d_{p_i}$ and $f_k=d_{p_k}$ respectively, where $k>i$. 
	
	For every $x\in \partial W_{t_i}\cap \partial W_{t_k}$, we have $d(p_i, x)=d(p_i, p_k)=t_i$. By assumption, we have $d(x,p_k)=t_k\le\epsilon t_i$. Thus, the comparison angle $\tilde{\angle} p_ix p_k$ is within $10\epsilon$ of $\frac{\pi}{2}$. By the good scale annulus structure of $A_{a_i}^{b_i}(p_i)$, one can find $p_i'\in A_{10b_i}^{11b_i}(p_i)$ such that the comparison angle $\tilde{\angle} p_ix p_i'\ge \pi-10\epsilon$. By splitting theory (or Lemma 5.6 in \cite{BGP}), we have $\left|\angle p_ix p_k-\frac{\pi}{2}\right|\le 50\epsilon$. 
	
	For $x\in \partial W_{t_i}\cap \partial X$, the result follows because $\nabla_x d_{\partial X}$ is always $\epsilon$-orthogonal to the almost splitting direction at $x$, which can be chosen to be $\nabla_x d_{p_i}$.
\end{proof}

The condition $0 < t_{i+1} \le \epsilon t_i$ can be imposed without loss of generality, as Theorem \ref{t:ann_cover} remains valid under the additional constraint $b_i \le \epsilon \cdot \diam(\Omega)$. Consequently, by Theorem \ref{t:ann_cover} and Lemma \ref{lem:ortho-scale}, we establish the following:

\begin{lemma}[Orthogonal Frame Covering]\label{t:cov-frame}
	Let $\left\{\partial W^1_{t_1}, \dots,\partial W^m_{t_m}\right\}$ be an $(m,\epsilon)$-orthogonal frame in $B_1$, with $1\le m <n$. There exists a covering of the corner space $\displaystyle Y=\cap_{i=1}^m\partial W^i_{t_i}$ by $\epsilon$-annuli $\mathcal C=\{W^{m+1,\ell}=W_{a_\ell}^{b_\ell}(q_\ell)\}$ with $q_\ell\in Y$ for every $\ell$, such that the following hold for every $1\le \ell\le |\mathcal C|$ and $a_\ell\le s_\ell\le b_\ell$.
	\begin{enumerate}
		\renewcommand{\labelenumi}{(\roman{enumi})}
		\item $|\mathcal C|\le N(n,\epsilon)$. 
		\item $\left\{\partial W^1_{t_1}, \dots,\partial W^m_{t_m}, \partial W^{m+1,\ell}_{s_\ell}\right\}$ is an $\displaystyle \left(m+1,\,\epsilon\right)$-orthogonal frame.
	\end{enumerate}
\end{lemma}

Let $Y$ be an $\epsilon$-orthogonal corner space. We would like to perturb the defining functions $f_i$ such that the intersection angles of the level sets become acute. While this can be accomplished for every point in $Y \setminus \partial X$, it is only possible for almost every point in $Y \cap \partial X$. To establish such a full-measure result, we need to provide a quantitative estimate of the singular sets on $Y$.

Let $\epsilon>0$ and $Y$ be an $(m,\epsilon)$-orthogonal corner space with defining functions $\{f_i\}$. By construction, for every $x\in Y$, there exists a map $u\colon X\to \dR^m$ that is $(m, \epsilon)$-splitting in a neighborhood of $x$, with the splitting directions $\nabla_x f_i$ being $\epsilon$-orthogonal to $Y$. Let $N\in [m, n-2]$ be an integer and let $\eta_1>0$ be small. Then for every $x\in Y\setminus \cS^N_{\eta_1}(X)$, there exists an $(N-m, \eta_1)$-splitting function $u_1\colon X\to \dR^{N-m}$ such that the pair $(u, u_1)\colon X\to \dR^N$ is an $(N, \eta_1)$-splitting function for $x$. For any sufficiently small $\eta_2>0$, by the Annulus Covering Theorem (Theorem \ref{t:ann_cover}) or the same proof as in Lemma 6.1 \cite{LiNab20}, there exists a sufficiently small $\epsilon \in (0, \eta_1)$ such that if we cover $\left(\cS^{N+1}_{\eta_2}(X)\setminus \cS^{N}_{\eta_1}(X)\right)\cap Y$ by balls $\mathcal B = \{B_{r_j}(x_j)\}$ with each $x_j\in \cS^{N+1}_{\eta_2}(X)\cap Y$, then for any $t\in \dR^N$, there are at most $C(n,\eta_2)$ balls in $\mathcal B$ that intersect $u_1^{-1}(t)\cap Y$. Following the same argument as in Section 6 \cite{LiNab20}, we obtain the following Hausdorff measure estimates.

\begin{lemma}\label{l:intrinsic singular est} 
	Let $n\ge N\ge m\ge 0$ be integers and $Y$ be an $(m,\epsilon)$-corner space. For any $\eta>0$, there exist $\epsilon=\epsilon(n,\eta)>0$ and $C(n,\epsilon)>0$ such that the following holds. 
	\begin{enumerate}
		\renewcommand{\labelenumi}{(\roman{enumi})}
		\item  $\cH^{N-m}\left(\cS^{N}_{\eta}(X)\cap Y\right)<C(n,\epsilon)$. Therefore, $\dim_{\cH}\left(\cS^{N}_\eta(X)\cap Y\right)\le N-m$ and $\dim_{\cH}(Y)=n-m$.
		\item $\cH^{n-m-2}\left(\cS^{n-1}_{\eta}(X)\cap Y\setminus\partial X\right)<C(n,\epsilon)$. Therefore, 
		$$ \dim_{\cH}\left(\cS_\eta(X)\cap Y\setminus\partial X\right)\le \dim_{\cH}(Y)-2. $$
	\end{enumerate}
\end{lemma}

By setting $N=n$ and $N=n-1$, respectively, we obtain the following Hausdorff measure estimates.

\begin{corollary}\label{c:vol-frame}
	Under the same assumptions as in Lemma \ref{l:intrinsic singular est}, let $k=\dim_{\cH}(Y)=n-m$. 
	\begin{enumerate}
		\renewcommand{\labelenumi}{(\roman{enumi})}
		\item  $\cH^{k}\left(Y\right)<C(n,\epsilon)\cdot\diam(Y)^{k}$.
		\item $\cH^{k-1}\left(Y\cap\partial X\right)<C(n,\epsilon)\cdot\diam(Y)^{k-1}$. 
	\end{enumerate}
\end{corollary}

\begin{remark}
	Indeed, through a more sophisticated analysis, one can show that $\cH^{k}\left(Y\right)\le C(n)\cdot\diam(Y)^{k}$ and $\cH^{k-1}\left(Y\cap\partial X\right)\le C(n)\cdot\diam(Y)^{k-1}$. We omit the proof here, as Corollary \ref{c:vol-frame} suffices for our purposes. 
\end{remark}

\begin{remark}\label{r:non-frame}
	According to the covering theorem above, on each $(m,\epsilon)$-corner space $Y$ with $\diam(Y)\le 1$, there are at most $C(n,\epsilon)$ points that do not belong to any $(m+1,\epsilon)$-frame. These points precisely correspond to the centers of the $\epsilon$-annuli in the covering of $Y$. Coupled with the local splitting structure within the annular regions, this implies that, away from a set of codimension 2, every point in $X\setminus\partial X$ admits an $(n,\epsilon)$-frame. 
\end{remark}

As mentioned earlier, the angle between $\nabla_xd_{p_i}$ and $\nabla_xd_{p_j}$ for $j>i$ is close to $\frac{\pi}{2}$ whenever $d(p_j,x)=t_j<\epsilon t_i$. To perturb this intersection angle, we need to push the level set toward $p_i$ along the direction $-\nabla_x d_{p_i}$. However, the function $-d_{p_i}$ cannot be used directly in the perturbation because it is not semi-concave. The following lemma resolves this issue. 

\begin{lemma}\label{l:opp-dist}
	There exists a constant $c(n)>0$ such that the following holds. Let $A_{a}^{b}(p)$ be an $(\epsilon, \sigma)$-good scale annulus. Then there exists a $\frac{c(n)}{b}$-concave function $\rho$ on $B_{\sigma^{-1}b}(p)$ satisfying the following properties:
	\begin{enumerate}
		\renewcommand{\labelenumi}{(\roman{enumi})}
		\setlength{\itemsep}{1pt}
		\item $|\rho(x)+d_p(x)|\le \epsilon d_p(x)$ for every $x\in A_{\sigma a}^{\sigma^{-1}b}(p)$. Consequently, for any $\sigma a\le s\le t\le \sigma^{-1}b$, we have
		\begin{enumerate}
			\renewcommand{\labelenumii}{(\alph{enumii})}
			\item $\partial {\mathcal W}_s(p)=\rho^{-1}(-s)\subseteq A_{(1-\epsilon^{1.5})s}^{(1+\epsilon^{1.5})s}(p)$;
			\item $A_{(1+\epsilon^{1.5})s}^{(1-\epsilon^{1.5})t}(p)\subseteq {\mathcal W}_{s,t}(p) = \{x\colon -t\le \rho(x)\le -s\}
			\subseteq A_{(1-\epsilon^{1.5})s}^{(1+\epsilon^{1.5})t}(p)$.
		\end{enumerate}
		\item $\big||\nabla_x\rho|-1\big|<\epsilon$ and $\big|\langle \nabla_x\rho, \nabla_x d_p\rangle+1\big|< \epsilon$ for every $x\in A_{\sigma a}^{\sigma^{-1}b}(p)\setminus\{p\}$.
	\end{enumerate}
\end{lemma}

\begin{proof}
	Let $r=\sigma^{-2} b$. Let $\mathcal{Q}=\{q_\alpha\}$ be a maximal collection of points in $A_r^r(p)$ such that $d(q_{\alpha_1}, q_{\alpha_2})\ge\delta_0 r$ for all $\alpha_1\neq\alpha_2$. For each $\alpha$, let $\mathcal{Q}^\alpha=\{q_\alpha^\beta\}$ be a maximal collection of points in $B_{\delta_0 r}(q_\alpha)\cap A_r^r(p)$ such that $d(q_\alpha^{\beta_1}, q_\alpha^{\beta_2})\ge\delta_1\delta_0 r$ for all $\beta_1\neq\beta_2$. Define 
	\begin{align}
		\rho_\alpha(x)=\frac{1}{|\mathcal{Q}^\alpha|}\sum_{q_\alpha^\beta\in\mathcal{Q}^\alpha} d_{q_\alpha^\beta}(x)-r
		\label{def:concave-func.e1}
	\end{align}
	and $\displaystyle \rho(x)=\min_\alpha \rho_\alpha(x)$. 
	
	(i) For any $q\in \partial B_{r}(p)$ and $x\in B_{\sigma^{-1}b}(p)$, we have
	\begin{align}
		d_{q_\alpha^\beta}(x)\ge d_p({q_\alpha^\beta})-d_p(x)=r-d_p(x).
		\label{l:cox-haull.e8}
	\end{align}
	Thus, $\rho_\alpha(x)\ge -d_p(x)$ for every $\alpha$, which implies 
	\begin{align}
		\rho(x)\ge -d_p(x).
		\label{l:cox-haull.e9}
	\end{align}
	On the other hand, for any $x\in A_{0.5a}^{2b}(p)$, the good scale annulus structure ensures there exists a $q_\alpha\in \mathcal{Q}$ such that $\tilde{\angle} xp q_\alpha \le 10(\delta_0+\epsilon^2)$. Moreover, 
	\begin{align}
		\tilde{\angle} xp q_\alpha^\beta \le 10(\delta_0+\epsilon^2+\delta_0\delta_1)\le 60\epsilon^2
		\label{l:cox-haull.e10}
	\end{align} 
	for every $q_\alpha^\beta\in\mathcal{Q}^\alpha$, provided $0<\delta_0\le 4\epsilon^2$. By triangle comparison, this implies
	\begin{align}
		d_{q_\alpha^\beta}(x)\le r-d_p(x)\cos(60\epsilon^2).
		\label{l:cox-haull.e11}
	\end{align}
	Thus, we have
	\begin{align}
		d_{q_\alpha^\beta}(x)-r \le -\cos(60\epsilon^2)d_p(x)
		\le -(1-4\epsilon^2)d_p(x).
		\label{l:cox-haull.e12}
	\end{align}
	Therefore,     
	\begin{align}
		\rho(x)\le \rho_\alpha(x)\le -(1-4\epsilon^2)d_p(x).
		\label{l:cox-haull.e13}
	\end{align}
	Combining (\ref{l:cox-haull.e9}) and (\ref{l:cox-haull.e13}), we obtain 
	\begin{align}
		|\rho(x)+d_p(x)|\le 4\epsilon^2 d_p(x).
		\label{l:cox-haull.e14}
	\end{align}
	In particular, $\rho^{-1}(-s)\subseteq A_{(1-4\epsilon^2)s}^{(1+4\epsilon^2)s}(p)$ for any $\sigma a\le s\le \sigma^{-1}b$.
	
	(ii) Now we show that $\big||\nabla_x\rho|-1\big|<10\epsilon^2$ for any $x\in A_{\sigma a}^{\sigma^{-1}b}(p)\setminus\{p\}$. Let $\alpha_0$ be an index for which $\rho(x)=\rho_{\alpha_0}(x)$. Then for any $y$, the triangle inequality yields 
	\begin{align}
		\rho(y)-\rho(x)&\le \rho_{\alpha_0}(y)-\rho_{\alpha_0}(x)\le d(x,y).
		\label{l:cox-haull.e50}
	\end{align}
	This implies that $|\nabla_x\rho|\le 1$ for every $x\in B_{\sigma^{-1}b}(p)$. 
	
	Let $y\neq x$ be a point on a geodesic connecting $p$ and $x$. Let $\alpha$ be an index for which $\rho(y)=\rho_\alpha(y)$. By the good scale structure, we have $\tilde{\angle} p y q_\alpha \ge \pi-2\epsilon^2-2\delta_0$. Therefore, $\tilde{\angle} p x q_\alpha \ge \pi-6\epsilon^2-6\delta_0$, and furthermore, $\tilde{\angle} p x q_\alpha^\beta \ge \pi-10\epsilon^2-10\delta_0$ for every $q_\alpha^\beta\in \mathcal{Q}^\alpha$. By Toponogov's comparison theorem, 
	\begin{align}
		\tilde{\angle} y x q_\alpha^\beta \ge \tilde{\angle} p x q_\alpha^\beta \ge \pi-10\epsilon^2-10\delta_0.
	\end{align} 
	Thus, 
	\begin{align}
		d_{q_\alpha^\beta}(y)\ge d_{q_\alpha^\beta}(x)+(1-20\epsilon^2)\cdot d(x,y).
	\end{align}
	Averaging this inequality over $q_\alpha^\beta\in \mathcal{Q}^\alpha$, we get 
	\begin{align}
		\rho(y)&=\rho_\alpha(y)\ge \rho_\alpha(x)+(1-20\epsilon^2)\cdot d(x,y) \notag \\
		&\ge \rho(x)+(1-20\epsilon^2)\cdot d(x,y).
		\label{l:cox-haull.e60}
	\end{align}
	This implies that $|\nabla_x\rho|\ge 1-20\epsilon^2$ for every $x\in A_{\sigma a}^{\sigma^{-1}b}(p)$. 
	
	Choosing $y=p$ in \eqref{l:cox-haull.e60}, we obtain
	\begin{align}
		\rho(p)\ge \rho(x)+(1-20\epsilon^2)\cdot d_p(x).
		\label{l:cox-haull.e70}
	\end{align}
	By the concavity of $\rho$, we have
	\begin{align}
		\langle \nabla_x\rho, \uparrow_x^p \rangle \ge \frac{\rho(p)-\rho(x)}{d_p(x)}\ge 1-20\epsilon^2.
	\end{align}
	Note that $\big|\langle \nabla_x d_p, \uparrow_x^p \rangle+1\big|<2\epsilon^2$. We conclude $\big|\langle \nabla_x\rho, \nabla_x d_p \rangle+1\big|< 100\epsilon^2$.
\end{proof}

Now we perturb an $\epsilon$-orthogonal frame into an a.e.-acute frame.

\begin{lemma}\label{l:frame.app.cute}
	Let $\mathcal{F} = \{\partial W^1_{t_1}(p_1), \partial W^2_{t_2}(p_2),\dots, \partial W^m_{t_m}(p_m)\}$ be an $(\epsilon, \delta, \delta_1, \delta_2, f_1, \dots, f_m)$-orthogonal frame, where $\delta, \delta_i$ are of the form $c(n)\epsilon$. Then there exists an a.e.-acute $(\epsilon, g_1, \dots, g_m)$-frame 
	$$\mathcal{G} = \{\partial \widetilde{W}^1_{t_1}(p_1), \partial \widetilde{W}^2_{t_2}(p_2),\dots, \partial \widetilde{W}^m_{t_m}(p_m)\}$$
	such that 
	\begin{enumerate}
		\renewcommand{\labelenumi}{(\roman{enumi})}
		\setlength{\itemsep}{1pt}
		\item  $W_{(1+\epsilon)a_i }^{(1-\epsilon)b_i}(q)\subseteq \widetilde{W}_{a_i}^{b_i}(q)\subseteq W_{(1-\epsilon)a_i}^{(1+\epsilon)b_i}(q)$,
		\item $\partial \widetilde{W}_{t}(q)\subseteq W_{(1-\epsilon)t}^{(1+\epsilon)t}(q)$ for every $t\in[a_i, b_i]$.
	\end{enumerate}	
	Furthermore, if the frame $\{\partial W^1_{t_1}(p_1), \partial W^2_{t_2}(p_2), \dots, \partial W^k_{t_k}(p_k)\}$ is a.e.-acute for some $k \le m$, then we can choose $g_i = f_i$ for all $1 \le i \le k$.
\end{lemma}

\begin{proof}
	We proceed by induction on $m$. The case of $m=1$ is trivial. Suppose the statement holds for some $1\le m\le n-1$. That is, we can assume that $\mathcal{F}_m = \{\partial W^1_{t_1}(p_1), \partial W^2_{t_2}(p_2),\dots, \partial W^m_{t_m}(p_m)\}$ is an $(\epsilon, \delta, \delta_1, \delta_2, f_1, \dots, f_m)$-a.e.-acute frame for  $\delta_1<\delta_2<0$, and 
	\begin{align}
		\left|\left\langle \nabla_x f_{m+1}, \nabla_xf_i \right\rangle\right|\le \tilde{\epsilon}
		\label{t:adj-angle.e4}
	\end{align}
	for every $1\le i\le m$ and $x\in\partial W_{t_{m+1}}\cap \partial W_{t_{i}}$. Here $\tilde{\epsilon}>0$ is of the form $c(n)\epsilon$. 
	
	For each good scale annulus $A_{a_i}^{b_i}(p_i)$, let $\rho_i(x)$ denote the approximation of $-d_{p_i}(x)$ defined in Lemma \ref{l:opp-dist}. It is clear that $\left|\left\langle \nabla_x\rho_i, \nabla_x d_{p_i}\right\rangle+1\right|< \epsilon$ for any $x\in Y=\bigcap_{i=1}^m \partial W_{t_i}^i(p_i)$. We approximate $f_{m+1}(x)$ by 
	\begin{align}
		g(x)= f_{m+1}(x)+30\tilde{\epsilon}\cdot\sum_{i=1}^m \left(\rho_i(x)- \rho_i(p_{m+1})\right) + C\tilde{\epsilon}\Big( d_{\partial X}(x)-d_{\partial X}(p_{m+1})\Big),
		\label{t:adj-angle.e22}
	\end{align}
	which is a perturbation of $f_{m+1}(x)$ by functions whose Hessians are bounded from above by $\frac{c(n)}{t_i}$ and $c(n)$ respectively. Here $C=10$ if $\partial X\neq\varnothing$, otherwise $C=0$. In the following we assume $\partial X\neq\varnothing$. The case of $\partial X=\varnothing$ follows similarly.

	It is easy to see that Definition \ref{d:cvx.ann.mfd} is satisfied and $(W_a^b(p_{m+1}), g)$ is an $(\epsilon, \sigma, 40\tilde{\epsilon})$-adjusted annulus. Therefore, $\widetilde{W}_{a_\ell}^{b_\ell}(p_{m+1})$, defined by $g$, is an adjusted $\epsilon$-good scale annulus, and (i) and (ii) are satisfied for $i=m+1$. 
	
	Now we prove the acuteness. Set $\delta_0=10(\epsilon+\delta+|\delta_1|+|\delta_2|)$. Let $t\in[a_{m+1}, b_{m+1}]$ and $x\in \partial W^{m+1}_t\cap Y$. For any $1\le i\le m$, since $(W^i,p_i)$ is an $(f_i,\epsilon,\sigma,\delta)$-annulus and $p_{m+1}\in Y\subseteq\partial W^i_{t_i}$, we have 
	\begin{align}
		\left|\left\langle \nabla_x \rho_i, \nabla_x f_i\right\rangle+1\right|\le \delta_0
		\label{t:adj-angle.e5}
	\end{align}
	for every $1\le i\le m+1$. By the inductive hypothesis, we have
	\begin{align}
		\left|\left\langle \nabla_x \rho_j, \nabla_x f_k\right\rangle\right|\le 2\delta_0
		\label{t:adj-angle.e6}
	\end{align}
	for $1\le j\neq k\le m$. 
	
	Combining (\ref{t:adj-angle.e22}) -- (\ref{t:adj-angle.e6}), we obtain
	\begin{align}
		\left\langle \nabla_xg, \nabla_x f_j\right\rangle
		& \le \left|\left\langle \nabla_x f_{m+1}, \nabla_x f_j\right\rangle\right|
		+ \left\langle \nabla_xg - \nabla_x f_{m+1}, \nabla_x f_j\right\rangle
		\notag\\     
		& \le \tilde{\epsilon}+30\tilde{\epsilon}\sum_{i=1}^m \left\langle \nabla_x \rho_i, \nabla_x f_j\right\rangle 
		+10\tilde{\epsilon}\left\langle \nabla_xd_{\partial X}, \nabla_x f_j\right\rangle
		\notag\\
		&\le \tilde{\epsilon}+30\tilde{\epsilon}\big(m\cdot 2\delta_0-(1-\delta_0)\big)+10\tilde{\epsilon}
		\notag\\
		&\le -10\tilde{\epsilon}, \label{t:adj-angle.e7.2}
	\end{align}
	provided $\delta_0<\frac{1}{100n}$. Similarly, 
	\begin{align}
		\left\langle \nabla_x g, \nabla_x f_j\right\rangle
		& \ge -\left|\left\langle \nabla_x f_{m+1}, \nabla_x f_j \right\rangle\right|
		+ \left\langle \nabla_xg - \nabla_x f_{m+1}, \nabla_x f_j \right\rangle
		\notag\\     
		& \ge -\tilde{\epsilon} + 30\tilde{\epsilon}\sum_{i=1}^m \left\langle \nabla_x  \rho_i, \nabla_x f_j\right\rangle-10\tilde{\epsilon}\left\langle \nabla_xd_{\partial X}, \nabla_x f_j\right\rangle
		\notag\\
		&\ge -\tilde{\epsilon}+30\tilde{\epsilon}\big(-m\cdot2\delta_0-(1+\delta_0)\big)-10\tilde{\epsilon}
		\notag\\
		&\ge -50\tilde{\epsilon}.
		\label{t:adj-angle.e7.1}
	\end{align}
	
	Let $\dim_{\cH}(Y)=k=n-m$. Note that $\dim_{\cH}(\partial \widetilde{W}^{m+1}\cap Y)=k-1$ and $\dim_{\mathcal{H}}(\partial \widetilde{W}^{m+1}\cap Y\cap \partial X)=k-2$. We claim that:
	\begin{enumerate}
		\item $|\left\langle \nabla_x f_{m+1}, \nabla_xd_{\partial X}\right\rangle|\le 4\epsilon$ for every $x\in \partial \widetilde{W}^{m+1}\cap \partial X$;
		\item For any $\epsilon_1 >0$, there exists $\epsilon=\epsilon(n,\epsilon_1)>0$ such that     
		\begin{align}
			\left||\nabla_xd_{\partial X}|-1\right|\le 10\epsilon
			\label{t:adj-angle.e7.3.0}
		\end{align}
		holds for every $x\in \left(\partial \widetilde{W}^{m+1}\cap Y\cap \partial X\right)\setminus \mathcal{S}^{n-2}_{\epsilon_1}(X)$. Moreover, (\ref{t:adj-angle.e7.3.0}) holds for $\mathcal{H}^{k-2}$-almost every $x\in \partial \widetilde{W}^{m+1}\cap Y\cap \partial X$.
	\end{enumerate}
	
	Let $x\in \partial \widetilde{W}^{m+1}\cap \partial X$. Then $B_r(x)$ is $2\epsilon$-splitting along $\nabla_x f_{m+1}$ for small $r>0$. Statement (1) follows from the fact that $\nabla_xd_{\partial X}$ is $\epsilon$-orthogonal to any splitting direction of $B_r(x)$. 
	
	For $x\notin \mathcal{S}^{n-2}_{\epsilon_1}(X)$, we know that $B_r(x)$ is $(n-1,\epsilon)$-splitting for small $r>0$ and the splitting directions are all $\epsilon$-tangent to $\partial X$.
	This implies that $\left||\nabla_xd_{\partial X}|-1\right|\le 10\epsilon$. The second statement in (2) follows from Lemma \ref{l:intrinsic singular est} (i): 
	\begin{align}
		\mathcal{H}^{k-3}\left(\mathcal{S}^{n-2}_{\epsilon_1}(X)\cap\left(\partial \widetilde{W}^{m+1}\cap Y\right)\right)<C(n,\epsilon).
	\end{align} 
	
	By Claims (1) and (2), for $\mathcal{H}^{k-2}$-almost every $x\in \partial \widetilde{W}^{m+1}\cap Y\cap \partial X$, we have 
	\begin{align}
		\left\langle \nabla_x g, \nabla_xd_{\partial X}\right\rangle
		& \ge -\left|\left\langle \nabla_x f_{m+1}, \nabla_xd_{\partial X} \right\rangle\right|
		+ \left\langle \nabla_xg - \nabla_x f_{m+1}, \nabla_xd_{\partial X} \right\rangle
		\notag\\     
		& \ge -\tilde{\epsilon} + 30\tilde{\epsilon}\sum_{i=1}^m \left\langle \nabla_x \rho_i, \nabla_xd_{\partial X}\right\rangle + 10\tilde{\epsilon}\left\langle \nabla_xd_{\partial X}, \nabla_xd_{\partial X}\right\rangle
		\notag\\
		&\ge -\tilde{\epsilon} - 30\tilde{\epsilon}(m\cdot 4\tilde{\epsilon}) + 10\tilde{\epsilon} (1-30\epsilon)
		\notag\\
		&\ge 5\tilde{\epsilon},
		\label{t:adj-angle.e7.3}
	\end{align}
	and  
	\begin{align}
		\left\langle \nabla_x g, \nabla_xd_{\partial X}\right\rangle
		& \le \left|\left\langle \nabla_x f_{m+1}, \nabla_xd_{\partial X} \right\rangle\right|
		+ \left\langle \nabla_xg - \nabla_x f_{m+1}, \nabla_xd_{\partial X} \right\rangle
		\notag\\     
		& \le \tilde{\epsilon} + 30\tilde{\epsilon}\sum_{i=1}^m \left\langle \nabla_x \rho_i, \nabla_xd_{\partial X}\right\rangle + 10\tilde{\epsilon}\left\langle \nabla_xd_{\partial X}, \nabla_xd_{\partial X}\right\rangle
		\notag\\
		&\le \tilde{\epsilon} + 30\tilde{\epsilon}(m\cdot 4\tilde{\epsilon}) + 10\tilde{\epsilon} (1+30\epsilon)
		\notag\\
		&\le 20\tilde{\epsilon}.
		\label{t:adj-angle.e7.4}
	\end{align}
\end{proof}

Combining Lemma \ref{t:cov-frame} and Lemma \ref{l:frame.app.cute}, we establish the following a.e.-acute frame covering theorem. 

\begin{theorem}[a.e.-acute Frame Covering]\label{t:cov-angle}
	Let $\left\{\partial W^1_{t_1}, \dots,\partial W^m_{t_m}\right\}$ be an a.e.-acute $(m,\epsilon)$-frame in $B_1$, where $1\le m <n$. There exists a covering of the corner space $\displaystyle Y=\bigcap_{i=1}^m\partial W^i_{t_i}$ by $\epsilon$-annuli $\mathcal{C}=\{W^{m+1,\ell}=W_{a_\ell}^{b_\ell}(q_\ell)\}$ with $q_\ell\in Y$ for each $\ell$, such that the following hold for every $1\le \ell\le |\mathcal{C}|$ and $a_\ell\le s_\ell\le b_\ell$:
	\begin{enumerate}
		\renewcommand{\labelenumi}{(\roman{enumi})}
		\item $|\mathcal{C}|\le N(n,\epsilon)$;
		\item $\left\{\partial W^1_{t_1}, \dots,\partial W^m_{t_m}, \partial W^{m+1,\ell}_{s_\ell}\right\}$ is an a.e.-acute $(m+1,\epsilon)$-frame.
	\end{enumerate}
\end{theorem}

\section{Integration over $\epsilon$-Frames}

By the assumptions of \hyperlink{thm:A}{Theorem A}, $M=B_2(p)\setminus\cS_\eta(X)$ is a smooth manifold with smooth boundary. We choose $\epsilon=\epsilon(n,\eta)>0$ to be sufficiently small such that Lemma \ref{l:intrinsic singular est} holds. A corner space $Y$, or its corresponding frame $\left\{\partial W^1_{t_1}, \dots,\partial W^m_{t_m}\right\}$, is said to be virtually smooth if every defining function $f_i$ is smooth on $W^i\setminus\cS_\eta(X)$. By the smoothness assumption of $X$, the semi-concave functions defined in Section \ref{sec:frame-covering} can always be perturbed to be smooth on $M$, thereby making the frames virtually smooth \cite{AKP08, GW79}. Therefore, \hyperlink{thm:A}{Theorem A} follows directly from the $m=0$ case of the following theorem.

\begin{theorem}\label{l:int.corner.mfd} 
	Let $\left\{\partial W^1_{t_1}, \dots,\partial W^m_{t_m}\right\}$ be a virtually smooth a.e.-acute $(m,\epsilon)$-frame for $0\le m \le n-2$, and let $\displaystyle Y=\bigcap_{i=1}^m\partial W^i_{t_i}$ be an a.e.-acute $\epsilon$-corner space. Define $K_Y^-(x)=\max\{-\min\{\sec_Y(x)\}, 1\}$ and let $k=\dim(Y)=n-m$. Then,
	\begin{align}
		\int_{Y\cap B_1(p_0)} scal_Y \,\mathrm{d}\mathrm{vol}_Y \le c(n,\epsilon)\cdot\left(\operatorname{diam}(Y)^{k-2}+\int_{Y\cap B_2(p_0)} K_Y^-\right).
		\label{l:int.corner.mfd.e0}
	\end{align}
\end{theorem}

\begin{proof}
	We proceed by induction on $k=\dim(Y)=n-m$ to prove (\ref{l:int.corner.mfd.e0}). Assume $k\ge 3$. We will defer the case of $k=2$ until later, as its proof is largely similar to that of the $k\ge 3$ case. We denote $Y\cap B_r(p_0)$ by $Y_r$, and simply write $Y_1$ as $Y$. 
	
	\subsection{Basic setup and estimates}
	
	By Theorem \ref{t:cov-angle}, there exists an adjusted $\epsilon$-annulus covering 
	$$\mathcal{C} =\left\{W^{m+1,\ell}=\left(W_{a_\ell}^{b_\ell}(p_\ell), g_\ell\right), \ell = 1,2,\dots,C(n,\epsilon)\right\}$$ 
	of $Y$ with $p_\ell\in Y$, such that $\left\{\partial W^1_{t_1}, \dots,\partial W^m_{t_m}, \partial W^{m+1,\ell}_t=g_\ell^{-1}(t)\right\}$ is an a.e.-acute $\displaystyle \left(m+1,\, \epsilon\right)$-frame. For simplicity, we fix $\ell$ and denote $p_\ell$ by $p$, $a_\ell$ by $a$, $b_\ell$ by $b$, and $g_\ell$ by $g$. Let $L_t=Y\cap g^{-1}(t)=Y\cap\partial W^{m+1,\ell}_t$ be an $(m+1)$-th stratum of the corner space. By the smoothness assumption of $X$, the semi-concave defining function $g$ can be perturbed to be smooth on $M$, thereby making $L_t$ virtually smooth \cite{AKP08, GW79}. 
	
	The intrinsic scalar curvature $scal_{L_t}(x)$ and the extrinsic scalar curvature $scal_{Y}(x)$ are related by the Gauss–Codazzi formula:
	\begin{align}
		scal_{L_t}
		&=\sum_{i\neq j}\sec_{L_t}(e_i\wedge e_j)
		\notag \\ &= \sum_{i\neq j}\sec_Y(e_i\wedge e_j) +G
		\label{l:int.corner.mfd.e7.1} \\
		&=scal_{Y}-2\Ric_Y(u,u)+G,
		\label{l:int.corner.mfd.e7}
	\end{align}
	where $k_i(x)$ are the principal curvatures of $L_t$ at $x$ in $Y$, $G(x)=\sum_{i\neq j}k_i(x)k_j(x)$, and $u=\frac{\nabla_xg}{|\nabla_xg|}$ is the unit normal vector.

Let $b'\in[b,2b]$ be a parameter to be determined later. Integrating (\ref{l:int.corner.mfd.e7}) over $Y\cap g^{-1}[a,b']$ and applying the coarea formula, we obtain 
\begin{align}
	\int_{Y\cap \,g^{-1}[a,b']}scal_{Y}=\int_{a}^{b'}\int_{L_t}\frac{scal_{L_t}}{|\nabla_xg|}+\int_{Y\cap \,g^{-1}[a,b']}(2\Ric_Y(u,u)-G).
	\label{l:int.corner.mfd.e7.5}
\end{align}

We use the boundary version of the Bochner formula (see \cite{Pet09} for the version without boundary):
\begin{align}
	\int_{Y\cap \,g^{-1}[a,b']}\Ric_Y(u,u)=\int_{Y\cap \,g^{-1}[a,b']}G+\int_{L_{a}}H-\int_{L_{b'}}H - \int_{\partial X\cap Y} H\left\langle u, \overset\rightarrow{n} \right\rangle,
	\label{l:int.corner.mfd.e8}
\end{align}
where $H(x)=\sum k_i(x)$ is the mean curvature of $L_t$ in $Y$, and $\overset\rightarrow{n}$ is the outward unit normal vector of $\partial X\cap Y$ in $Y$. Substituting (\ref{l:int.corner.mfd.e8}) into (\ref{l:int.corner.mfd.e7.5}), we get 
\begin{align}
	\int_{Y\cap \,g^{-1}[a,b']}scal_{Y}
	& =\int_{a}^{b'}\int_{L_t}\frac{scal_{L_t}}{|\nabla_x g|}+\int_{Y\cap \,g^{-1}[a,b']}G
	\notag \\
	&+2\int_{L_{a}}H-2\int_{L_{b'}}H
	- 2\int_{\partial X\cap Y} H\left\langle u, \overset\rightarrow{n} \right\rangle.
	\label{l:int.corner.mfd.e9}
\end{align}

We will bound the right-hand side in terms of $K_Y^-$, using the good scale structure on $W_{a}^{b}(p)$. First, by (\ref{l:int.corner.mfd.e7.1}), we have 
\begin{align}
	G \le scal_{L_t}+n^2 K_Y^-.
	\label{l:int.corner.mfd.e9.5}
\end{align}
Define $scal_{L_t}^\pm(x)=\max\{\pm scal_{L_t}(x),0\}$. It is straightforward to see that 
\begin{align}
	\int_{L_t}\frac{scal_{L_t}}{|\nabla_x g|} 
	&= \int_{L_t}\frac{scal_{L_t}^+ - scal_{L_t}^- }{|\nabla_x g|} 
	\le 2\int_{L_t}scal_{L_t}^+ - \frac{1}{2} \int_{L_t}scal_{L_t}^-
	\notag \\
	& = 2\int_{L_t}scal_{L_t} + \frac{3}{2} \int_{L_t}scal_{L_t}^-.
\end{align}
Therefore, 
\begin{align}
	\int_{Y\cap \,g^{-1}[a,b']}scal_{Y}
	&\le 3\int_{a}^{b'}\int_{L_t}scal_{L_t}+2n^2\int_a^{b'}\int_{L_t} K_{L_t}^- 
	+m^2\int_{g^{-1}[a,b']}K_Y^-
	\notag \\
	&+2\int_{L_{a}}H-2\int_{L_{b'}}H - 2\int_{\partial X\cap Y} H\left\langle u, \overset\rightarrow{n} \right\rangle.
	\label{l:int.corner.mfd.e9.1}
\end{align}

\subsection{Upper bound of principal curvatures}

To estimate the principal curvatures $k_i$, we start with the following formula. 
Let $h$ be a function and $v$ be a vector tangent to $L_t$ at $x$. Then 
\begin{align}
	\operatorname{Hess}^Y_h(v,v)
	= \operatorname{Hess}^M_h(v,v) - \left\langle \RNum{2}^M_Y(v,v), \nabla^M_x h\right\rangle,
	\label{l:int.corner.mfd.e20}
\end{align}
where $\RNum{2}$ is defined with respect to the outward normal direction. Let $h=f_i$ be a defining function for $Y$. Then $\operatorname{Hess}^Y_h(v,v)=0$, and the components of $\RNum{2}^M_Y(v,v)$ are given by
\begin{align}
	\left\langle \RNum{2}^M_Y(v,v), \nabla^M_x f_i\right\rangle = \operatorname{Hess}^M_{f_i}(v,v).
	\label{l:int.corner.mfd.e20.1}
\end{align}
Substituting this into (\ref{l:int.corner.mfd.e20}), we obtain the Hessian decomposition formula:
\begin{align}
	\operatorname{Hess}^Y_h(v,v)
	= \operatorname{Hess}^M_h(v,v) - \sum_{i=1}^m \operatorname{Hess}^M_{f_i}(v,v)\cdot \left\langle \frac{\nabla^M_x f_i}{|\nabla^M_x f_i|^2}, \nabla^M_x h\right\rangle.
	\label{l:int.corner.mfd.e20.2}
\end{align}
Let $h=g$ and $u=\nabla_xg/|\nabla_x g|$. By (\ref{l:int.corner.mfd.e20.2}), we have
\begin{align}
	\left\langle \nabla_v u, v\right\rangle
	& = \frac{\left\langle \nabla_v \nabla_x g, v\right\rangle}{|\nabla_x g|}
	= \frac{\operatorname{Hess}^Y_g(v,v)}{|\nabla_x g|}
	\notag\\
	& = \frac{\operatorname{Hess}^M_g(v,v)}{|\nabla_x g|}
	- \sum_{i=1}^m \operatorname{Hess}^M_{f_i}(v,v)\cdot  \left\langle \frac{\nabla^M_x f_i}{|\nabla^M_x f_i|^2}, \frac{\nabla^M_x g}{|\nabla^M_x g|}\right\rangle.
	\label{l:int.corner.mfd.e20.3}
\end{align}

Note that 
\begin{enumerate}
	\item any function $h\in\{f_i, g\}$ is $(5/h)$-concave and $\epsilon h$-close to the corresponding distance function;
	\item $-1\le \left\langle \nabla^M_x f_i, \nabla^M_x g\right\rangle<0$.
\end{enumerate}
We have 
\begin{align}
	\left\langle \nabla_v u, v\right\rangle \le \frac{6n}{t}.
	\label{l:int.corner.mfd.e22}
\end{align}
Then
\begin{align}
	k_i\le\frac{6n}{t} 
	\label{l:int.corner.mfd.e23}
\end{align} 
and
\begin{align}
	H(x)\le \frac{6n^2}{t},
	\label{l:int.corner.mfd.e23-m}
\end{align} 
for a smooth point $x\in L_t$. By Corollary \ref{c:vol-frame},
\begin{align}
	\int_{L_a} H \le c(n,\epsilon)\cdot a^{k-2}.
	\label{l:int.corner.mfd.e23-m.2}
\end{align}

Furthermore, for a smooth point $x\in\partial X\cap Y$, note that $\overset\rightarrow{n} = \overset\rightarrow{N}/|\overset\rightarrow{N}|$, where 
\begin{align}
	\overset\rightarrow{N}
	= \sum_{i=1}^m \frac{\left\langle\nabla d_{\partial X}, \nabla f_i \right\rangle}{|\nabla f_i|^2}\cdot \nabla f_i - \nabla d_{\partial X}.
\end{align}
For $\cH^{k-2}$-almost every $x\in \partial X\cap Y$, 
\begin{align}
	\left\langle u, \overset\rightarrow{N} \right\rangle
	& = \sum_{i=1}^m \frac{\left\langle\nabla d_{\partial X}, \nabla f_i \right\rangle}{|\nabla f_i|^2}\cdot \frac{\left\langle \nabla g, \nabla f_i \right\rangle}{|\nabla g|} - \frac{\left\langle \nabla g,\nabla d_{\partial X}\right\rangle}{|\nabla g|} \le 0
\end{align}
due to the definition of the a.e.-acute frame. Define $H^\pm(x)=\max\{\pm H(x),0\}$. Together with (\ref{l:int.corner.mfd.e23-m}) and Corollary \ref{c:vol-frame}, we have 
\begin{align}
	-\int_{\partial X\cap Y} H\left\langle u, \overset\rightarrow{n} \right\rangle \le \int_{\partial X\cap Y} H^+\le c(n,\epsilon)\cdot b^{k-2}.
	\label{l:int.corner.mfd.e23-m.3}
\end{align}

Substituting (\ref{l:int.corner.mfd.e23-m.2}) and (\ref{l:int.corner.mfd.e23-m.3}) into (\ref{l:int.corner.mfd.e9.1}), we obtain 
\begin{align}
	\int_{Y\cap \,g^{-1}[a,b']}scal_{Y}
	&\le 3\int_{a}^{b'}\int_{L_t}scal_{L_t}+2n^2\int_a^{b'}\int_{L_t} K_{L_t}^- \notag \\
	&+ m^2\int_{g^{-1}[a,b']}K_Y^-+c(n,\epsilon)\cdot b^{k-2}-2\int_{L_{b'}}H.
	\label{l:int.corner.mfd.e9.6}
\end{align}

\subsection{Lower bound of mean curvature integral}

We now seek a value $b'\in[b,2b]$ such that the integral $\dsp\int_{L_{b'}}H$ is bounded from below. By Lemma \ref{l:intrinsic singular est} (ii), we have 
\begin{align}
	\dim_\cH\left(\cS(X) \cap L_t\setminus\partial X\right)\le \dim_{\cH}(L_t) -2.
\end{align} 
Since $B_2(p_0)\setminus \cS_{\eta}(X)$ is smooth, it follows from the preceding estimate that the only effective boundary term in the first variation formula for $\Vol{L_t}$ is the integral over $L_t\cap \partial X$:
\begin{align}
	\frac{d(\Vol{L_t})}{dt}
	&=\int_{L_t}\frac{H}{|\nabla_x g|}
	+\int_{\partial L_t\cup(\cS_\eta(X)\cap L_t)} \frac{1}{|\nabla_xg|\sin\theta}\left\langle V,\, \frac{\nabla_xg}{|\nabla_xg|}\right\rangle \notag
	\\
	&=\int_{L_t}\frac{H}{|\nabla_x g|}
	+\int_{L_t\cap \,\partial X} \frac{1}{|\nabla_xg|\sin\theta}\left\langle V,\, \frac{\nabla_xg}{|\nabla_xg|}\right\rangle.
	\label{l:int.corner.mfd.e9.6.5}
\end{align}
Here, $\theta$ denotes the angle of intersection between $g^{-1}(t)$ and $\partial X$, which is close to $\frac\pi2$ almost everywhere. The vector field $V$ is given by the linear combination 
$$V=-c_0\cdot \frac{\nabla_xd_{\partial X}}{|\nabla_xd_{\partial X}|} + \sum_{i=1}^m c_i\cdot \frac{\nabla_xf_i}{|\nabla_xf_i|},
$$
where the coefficients satisfy $0<c_i\le 1$. By Corollary \ref{c:vol-frame}, we have $\cH^{k-2}(L_t\cap \,\partial X)\le c(n,\epsilon)\cdot t^{k-2}$. Using the decomposition $H=H^+-H^-$ and the pointwise upper bound $H(x)\le \frac{6n^2}{t}$ from \eqref{l:int.corner.mfd.e23-m}, we obtain 
\begin{align}
	\frac{d(\Vol{L_t})}{dt}
	&\le c(n,\epsilon)\cdot t^{k-2} + \int_{L_t}\frac{H}{|\nabla_x g|}
	\notag \\
	&\le c(n,\epsilon)\cdot t^{k-2} + 2\int_{L_t}H^++\frac{1}{2}\int_{L_t}(-H^-)
	\notag \\
	&\le c(n,\epsilon)\cdot t^{k-2}+\frac{1}{2}\int_{L_t}H.
	\label{l:int.corner.mfd.e9.6.6}
\end{align}
By Corollary \ref{c:vol-frame}, we have $\Vol{L_b}\le c(n,\epsilon)b^{k-1}$. Therefore, 
$$\Vol{L_{2b}}-\Vol{L_b}\ge -c(n,\epsilon)b^{k-1}.$$
This implies that there exists a $b'\in[b,2b]$ such that
$$\frac{d(\Vol{L_t})}{dt}\bigg\vert_{t=b'}\ge -c(n,\epsilon)\cdot b^{k-2}.
$$
Combining this with \eqref{l:int.corner.mfd.e9.6.6}, we deduce that
\begin{align}
	\int_{L_{b'}}H\ge -c(n,\epsilon)\cdot b^{k-2}.
	\label{l:int.corner.mfd.e9.7}
\end{align} 
Substituting this result into \eqref{l:int.corner.mfd.e9.6}, we arrive at 
\begin{align}
	&\int_{Y\cap \,g^{-1}[a,b']}scal_{Y} 
	\notag \\
	&\quad \le 3\int_{a}^{b'}\int_{L_t}scal_{L_t}+2n^2\int_a^{b'}\int_{L_t} K_{L_t}^-+c(n) \cdot \int_{g^{-1}[a,b']}K_Y^-+c(n,\epsilon)\cdot b^{k-2},
	\label{l:int.corner.mfd.e11}
\end{align}
for some $b'\in[b,2b]$.

\subsection{Summing up over the covering}

Applying the inductive hypothesis to the a.e.-acute $\left(m+1,\,\epsilon\right)$-frame $\left\{\partial W^1_{t_1}, \dots,\partial W^m_{t_m}, \partial W^{m+1,\ell}_t\right\}$, we have
\begin{align}
	&\int_{L_t}scal_{L_t} 
	\le  c\left(n,\epsilon\right)\cdot\left(\diam(L_t)^{k-3}+\int_{L_t\cap B_2(p_0)} K_{L_t}^-\right),
	\label{l:int.corner.mfd.e5}
\end{align}
where $K_{L_t}^-(x)=\max\{-\min\{\sec_{L_t}(x)\},1\}$. Integrating \eqref{l:int.corner.mfd.e5} over $t\in[a,b']\subseteq[a,2b]$ yields 
\begin{align}
	&\int_{a}^{b'}\int_{L_t}scal_{L_t}dt 
	\le  c\left(n,\epsilon\right)\cdot \left(b^{k-2}+\int_a^{b'}\int_{L_t\cap B_2(p_0)} K_{L_t}^-\right).
	\label{l:int.corner.mfd.e6}
\end{align}
Consequently, by \eqref{l:int.corner.mfd.e11},
\begin{align}
	\int_{Y\cap\, g^{-1}[a,b']}scal_{Y}&\le c\left(n,\epsilon\right)\cdot \left(b^{k-2}+\int_a^{b'}\int_{L_t\cap B_2(p_0)} K_{L_t}^-+\int_{g^{-1}[a,b']}K_Y^-\right),
	\label{l:int.corner.mfd.e11.1}
\end{align}
for some $b'\in[b,2b]$.

It remains to establish an upper bound for $\dsp\int_a^{b'}\int_{L_t\cap B_2(p_0)} K_{L_t}^-$ in terms of $\dsp\int_{Y_2\cap\, g^{-1}[a,b']} K_Y^-$. Applying the Gauss–Codazzi formula once more, 
\begin{align}
	K_{L_t}^-\le K_{Y}^-+2\cdot(H^-+m\cdot\max\{k_i,0\})\cdot \max\{k_i,0\}.
	\label{l:int.corner.mfd.e6.1}
\end{align}
Since $\big||\nabla_x g|-1\big|<10\epsilon$, we obtain 
\begin{align}
	\int_{L_t} H^-
	&\le 2\int_{L_t}\frac{H^-}{|\nabla_x g|}
	= 2\int_{L_t}\frac{H^+}{|\nabla_x g|}
	-2\int_{L_t}\frac{H}{|\nabla_x g|}
	\notag \\
	&\le c(n,\epsilon)\cdot t^{k-2}-8\cdot \frac{d(\Vol{L_t})}{dt}.
\end{align}
The final inequality follows from \eqref{l:int.corner.mfd.e23} and \eqref{l:int.corner.mfd.e9.6.6} for $t\in[a,2b]$. Integrating this with respect to $t$, we find
\begin{align}
	\int_{Y_2\cap\, g^{-1}[a,b']} (H^-\cdot\max\{k_i,0\})
	\le c(n)\cdot b^{k-2}.
\end{align}
Combining this result with \eqref{l:int.corner.mfd.e6.1}, we obtain
\begin{align}
	\int_a^{b'}\int_{L_t\cap B_2(p_0)}K_{L_t}^-\le \int_{Y_2\cap\, g^{-1}[a,b']}K_{Y}^-+c(n)\cdot b^{k-2}.
\end{align}
Substituting this into \eqref{l:int.corner.mfd.e11.1}, we have
\begin{align}
	&\int_{Y\cap \,g^{-1}[a,b']}scal_{Y} 
	\le  c\left(n,\epsilon\right)\cdot \left(b^{k-2}+\int_{Y_2\cap\, g^{-1}[a,b']} K_{Y}^-\right).
\end{align}

Given that $scal_Y(x)\ge -m^2\cdot K_Y^-(x)$ and $b'\in[b,2b]$, we deduce that 
\begin{align}
	\int_{Y\cap \,g^{-1}[a,b]}scal_{Y} 
	&\le  2\cdot c\left(n,\epsilon\right)\cdot \left(b^{k-2}+\int_{Y_2\cap\, g^{-1}[a,b']} K_{Y}^-\right)
	\notag \\
	&\le 2\cdot c\left(n,\epsilon\right)\cdot \left(b^{k-2}+\int_{Y_2\cap\, g^{-1}[a,2b]} K_{Y}^-\right).
\end{align}
Summing the above inequality over all $\ell=1,2,\dots,|\mathcal C|$, we conclude 
\begin{align}
	\int_{Y}scal_{Y} 
	&\le c\left(n,\epsilon\right)\cdot \left(b^{k-2}+\int_{Y_2} K_{Y}^-\right).
	\label{l:int.corner.mfd.e50}
\end{align}

\subsection{The case of dimension 2}

The 2-dimensional case follows directly from the 2D reduction of the Bochner formula \eqref{l:int.corner.mfd.e8} and the finite number of connected components of the corner spaces. Suppose $m=n-2$ and $\dim(Y)=2$. We cover $Y$ in the same manner as above. Then for any connected component $Y^{(i)}$ of $Y$, since $\Ric_Y(u,u) = K_Y$, the cross term $G=0$, and the mean curvature $H$ reduces to the geodesic curvature $k_g$, formula \eqref{l:int.corner.mfd.e8} becomes:
\begin{align}
	\int_{Y^{(i)}\cap \,g^{-1}[a,b']} K_Y=\int_{L_{a}}k_g-\int_{L_{b'}}k_g - \int_{\partial X\cap Y^{(i)}} k_g\left\langle u, \overset\rightarrow{n} \right\rangle.
	\label{l:int.corner.mfd.e60}
\end{align}
Here $k_g$, replacing the mean curvature $H$, denotes the geodesic curvature of $L_t=Y\cap g^{-1}(t)$ in $Y$. Using the same estimates as in \eqref{l:int.corner.mfd.e23-m.2}, \eqref{l:int.corner.mfd.e23-m.3}, and \eqref{l:int.corner.mfd.e9.7}, we obtain 
\begin{align}
	\int_{Y^{(i)}\cap \,g^{-1}[a,b']} K_Y\le c(n,\epsilon).
	\label{l:int.corner.mfd.e61}
\end{align}
Summing this over the good scale annulus covering, we obtain 
\begin{align}
	\int_{Y^{(i)}} K_Y\le c(n,\epsilon).
	\label{l:int.corner.mfd.e62}
\end{align}

We now show that $Y$ has at most $c(n)$ connected components. This follows from the fact that $Y$ is $\frac{1}{100}\diam(Y)$-connected. That is, for any $z\in Y$, the intersection $B_{t_{n-2}/100}(z)\cap Y$ is path-connected in $Y$. Let $p_i$ be the center of $W^i$. For any $x, y\in B_{t_{n-2}/100}(z)\cap Y$, let $\gamma$ be a geodesic in $X$ connecting them. By the almost metric cone structure, we have $\gamma\subset A_{\frac12t_{n-2}}^{2t_{n-2}}(p_i)$ for each $1\le i\le n-2$. Note that $\displaystyle\bigcap_{i=1}^{n-2}A_{\frac12t_{n-2}}^{2t_{n-2}}(p_i)$ is homeomorphic to $Y\times [\frac12 t_{n-2}, 2t_{n-2}]^{n-2}$. The geodesic $\displaystyle\gamma\subset \bigcap_{i=1}^{n-2}A_{\frac12t_{n-2}}^{2t_{n-2}}(p_i)$ can be deformed into a continuous curve connecting $x$ and $y$ in $Y$. 

\end{proof}

\section{Examples of Fractal Curvature Distribution}

In this section, we prove Theorem \ref{thm:B} by construction.

Let $\mathcal{K} \subset [-0.5, 0.5]$ be a perfect set. Using standard constructions, we can define a continuous, monotonically increasing, and surjective map $\phi: [-0.5,0.5] \to [0,1]$ that is locally constant on $[-0.5, 0.5] \setminus \mathcal{K}$ and non-differentiable on $\mathcal{K}$. We define a continuous non-increasing function $f: \mathbb{R} \to [-1, 1]$ as follows:
\begin{equation}
	f(x) = 
	\begin{cases} 
		1 & x < -0.5 \\
		1 - 2\phi(x) & x \in [-0.5, 0.5] \\
		-1 & x > 0.5.
	\end{cases}
\end{equation}
Note that $f(x)$ is locally constant on $\mathbb{R} \setminus \mathcal{K}$. Let $g(x) = \int_0^x f(t) \, dt$. Then $g(x)$ possesses the following properties:
\begin{enumerate}
	\renewcommand{\labelenumi}{(\arabic{enumi})}
	\setlength{\itemsep}{1pt}
	\item $g$ is concave because its derivative $f$ is non-increasing.
	\item $g$ is locally linear on $\mathbb{R} \setminus \mathcal{K}$.
	\item $g(x)$ is linear with a slope of $1$ on $(-\infty, -0.5)$ and a slope of $-1$ on $(0.5, \infty)$.
\end{enumerate}

Let $F: \mathbb{R}^2 \to \mathbb{R}$ be the concave function defined by $F(x, y) = g(x) + g(y)$, and define the surface 
$$\Sigma = \{(x,y, F(x,y))\}$$ 
as the graph of $F(x, y)$. Let $X=\Sigma\cap \{z\ge -1\}\in\Alex^2(0)$ be the desired Alexandrov space. The boundary $\partial X$ corresponds to the level set $g(x) + g(y) = -1$. At any point $(x,y, F(x,y))$ where $x \notin \mathcal{K}$ or $y \notin \mathcal{K}$, $X$ is locally isometric to the product of a continuous curve and a small interval. Consequently, we have the following properties:
\begin{enumerate}
	\renewcommand{\labelenumi}{(\arabic{enumi})}
	\setlength{\itemsep}{1pt}
	\item $X\setminus\cN(X)$ is intrinsically flat.  
	\item $\cN(X)\setminus \partial X=\{(x,y, F(x,y)) \mid x,y \in \mathcal{K}\}$, which is bi-Lipschitz equivalent to $\mathcal{K}\times\mathcal{K}$.
\end{enumerate}
Clearly, $X$ is smoothable, since the concave function $F(x,y)$ is smoothable.

	

We now calculate the curvature measure of $X$ on $\cN(X)\setminus\partial X$ using the Gauss-Bonnet Theorem. Let $a,b\in\mathds R\setminus\mathcal{K}$ with $a< b$, and consider the domain $\Omega_{a,b}=\{(x,y,F(x,y))\colon (x,y)\in (a,b) \times (a,b) \}$. The closure $\overline{\Omega}_{a,b}$ is a curved quadrilateral on the surface with vertices at $A(a,a)$, $B(b,a)$, $C(b,b)$, and $D(a,b)$. Define 
\begin{equation}
	\Psi(u,v) = \arctan\left( \frac{uv}{\sqrt{1+u^2+v^2}} \right).
\end{equation}
Geometrically, $\Psi(u,v)$ captures the intersection angle of the coordinate curves. For the surface $\Sigma$, the angle $\theta$ between the curves $x=\text{const}$ and $y=\text{const}$ satisfies $\cot\theta = \frac{f(x)f(y)}{\sqrt{1+f(x)^2+f(y)^2}}$. By setting $u=f(x)$ and $v=f(y)$, we obtain $\theta = \frac{\pi}{2} - \Psi(u,v)$. Meanwhile, $\Psi(u,v)$ also serves as an explicit antiderivative for the geodesic curvature along these boundary curves. Direct computation shows that the total geodesic curvature along the four sides of $\partial\Omega_{a,b}$ is
\begin{equation}
	\sum_{i=1}^4 I_i = 4\Psi(f(a), f(b)) - 2\Psi(f(a), f(a)) - 2\Psi(f(b), f(b)).
	\label{curv.edge.e30}
\end{equation}
The sum of the exterior angles $\alpha_i$ at the four vertices $A$, $B$, $C$, and $D$ is 
\begin{align}
	\sum_{i=1}^4 \alpha_i &= \alpha_{aa} + \alpha_{ba} + \alpha_{bb} + \alpha_{ab} \notag \\
	&= \left(\frac{\pi}{2} + \Psi(f(a), f(a))\right) + \left(\frac{\pi}{2} - \Psi(f(a), f(b))\right) \notag \\
	&\quad + \left(\frac{\pi}{2} + \Psi(f(b), f(b))\right) + \left(\frac{\pi}{2} - \Psi(f(a), f(b))\right) \notag \\
	&= 2\pi + \Psi(f(a), f(a)) + \Psi(f(b), f(b)) - 2\Psi(f(a), f(b)).
	\label{curv.corner.e30}
\end{align}   
The total boundary curvature is thus
\begin{align}
	\oint_{\partial \Omega_{a,b}} k_g \, ds 
	& =\sum I_i + \sum \alpha_i 
	\\
	&= \big[4\Psi(f(a), f(b)) - 2\Psi(f(a), f(a)) - 2\Psi(f(b), f(b))\big]  
	\notag \\
	& \qquad + \big[ 2\pi + \Psi(f(a), f(a)) + \Psi(f(b), f(b)) - 2\Psi(f(a), f(b)) \big] \notag \\
	&= 2\pi + 2\Psi(f(a), f(b)) - \Psi(f(a), f(a)) - \Psi(f(b), f(b)).
	\label{curv.bdy.e10}
\end{align}

By the Gauss-Bonnet Theorem, the curvature measure is given by
\begin{equation}
	\mu(\Omega_{a,b})=\Psi(f(a), f(a)) + \Psi(f(b), f(b)) - 2\Psi(f(a), f(b)).
	\label{eq:total_curvature}
\end{equation}
In particular, by setting $a=-1$ and $b=1$, we can compute the total curvature on $\cN(X)\setminus\partial X$. In this case, 
\begin{equation}
	f(a) = f(-1) = 1, \quad f(b) = f(1) = -1.
\end{equation}
Substituting these values into the total curvature formula \eqref{eq:total_curvature}, we obtain:
\begin{equation}
	\mu(\Omega_{-1,1}) = \Psi(1, 1) + \Psi(-1, -1) - 2\Psi(1, -1),
\end{equation}
where
\begin{align}
	\Psi(1, 1) &= \arctan\left( \frac{1 \cdot 1}{\sqrt{1+1^2+1^2}} \right) = \arctan\left( \frac{1}{\sqrt{3}} \right) = \frac{\pi}{6}, \\
	\Psi(-1, -1) &= \arctan\left( \frac{(-1)(-1)}{\sqrt{1+(-1)^2+(-1)^2}} \right) = \arctan\left( \frac{1}{\sqrt{3}} \right) = \frac{\pi}{6}, \\
	\Psi(1, -1) &= \arctan\left( \frac{1 \cdot (-1)}{\sqrt{1+1^2+(-1)^2}} \right) = \arctan\left( -\frac{1}{\sqrt{3}} \right) = -\frac{\pi}{6}.
\end{align}
Combining these results, we conclude:
\begin{equation}
	\mu(\cN(X)\setminus \partial X) = \mu(\Omega_{-1,1}) = \frac{\pi}{6} + \frac{\pi}{6} - 2\left( -\frac{\pi}{6} \right) = \frac{2\pi}{3}.
\end{equation}

	We have established that $\mu(\cN(X)\setminus \partial X)>0$. However, this alone is not sufficient to conclude that $\mu|_{\cN(X)\setminus \partial X}$ is equivalent to $\cH^{2\alpha}$, where $\alpha=\dim_{\cH}(\cK)$. Observe that the mixed partial derivative is
	\begin{equation}
		\frac{\partial^2 \Psi}{\partial u \, \partial v} 
		= \frac{\partial}{\partial v} \left( \frac{v}{(1+u^2)\sqrt{1+u^2+v^2}} \right) = \frac{1}{(1+u^2+v^2)^{3/2}}>0.
	\end{equation}
	Geometrically, the term $\frac{1}{(1+u^2+v^2)^{3/2}} \, du \, dv$ coincides exactly with the spherical area element of the Gauss map parameterized by the gradients $(u, v)$. Consequently, the curvature measure can be expressed as a double integral over the gradient space:
	\begin{align}
		\mu(\Omega_{a,b}) & = \int_{f(b)}^{f(a)} \int_{f(b)}^{f(a)} \frac{\partial^2 \Psi}{\partial u \partial v} \, du \, dv
		= \int_{f(b)}^{f(a)} \int_{f(b)}^{f(a)} \frac{1}{(1+u^2+v^2)^{3/2}} \, du \, dv.
	\end{align}
	By the definition of $f(x)$, if the perfect set $\cK$ is Ahlfors regular, we have  
	\begin{equation}
		\mu(\Omega_{a,b}) >0 
		\quad \Longleftrightarrow \quad 
		f(a)-f(b) >0
		\quad \Longleftrightarrow \quad 
		\cH^{\alpha}([a,b]\cap\cK)>0.
		\label{curv.int.e30}
	\end{equation}

\end{document}